\documentclass{amsart}
\usepackage{header}


\title{The Parity of Invariant Characteristics}

\author[Linden Disney-Hogg]{Linden Disney-Hogg \orcidA{}}
\address{
School of Mathematics \\ Imperial College London \\ 
London, U.K.
}
\email{a.disneyhogg@imperial.ac.uk}

\thanks{\textbf{Acknowledgements.} LDH is grateful to Harry Braden for discussions on the proofs in the paper.}

\thanks{\textbf{Statements and Declarations.} The authors have no relevant financial or non-financial interests to disclose.}

\begin{document}

\begin{abstract}
    We demonstrate a method to prove that a theta characteristic on Riemann surface $S$ invariant under the action of $G \leq \mathrm{Aut}(S)$ has even parity by considering the representation theory of the double cover $\tilde{G} = 2 \cdot G$. We apply this to certain Hurwitz curves, and prove a recent conjecture of Broughton \& Disney-Hogg. 
\end{abstract}

\maketitle

\section{Introduction}

A \textit{theta characteristic} on a (smooth, compact, connected) Riemann surface $S$ is a line bundle $L \to S$ such that $L^{\otimes 2}$ is isomorphic to $K_S$ the canonical bundle. The \textit{parity} of a theta characteristic $q(L) := \dim \Gamma(L) \mod 2$($=\dim H^0(S, L) \mod 2$) is an important invariant under holomorphic deformations of $L \to S$ \cite{Atiyah1971}. Given a symplectic basis for the homology $H_1(S, \mathbb{Z})$, one may express theta characteristics as binary vectors $(\bm{u}, \bm{v}) \in \mathbb{Z}^{2g}$ where $g=g(S)$ is the genus of the surface, whereby the parity is simply $q = \bm{u} \cdot \bm{v}$ \cite{Johnson1980}. Historically, odd and even theta characteristics arose from the consideration of odd and even theta functions on the Jacobian of $S$; the modern study was given new impetus when Atiyah showed that theta characteristics are equivalent to spin structures on a Riemann surface \cite{Atiyah1971}. For a more complete history of theta characteristics and their modern application, see \cite{Farkas2012}.

The number of theta characteristics on a curve of genus $g$ is $2^{2g}$, $2^{g-1}(2^g-1)$ odd and $2^{g-1}(2^g+1)$ even, which are permuted by the induced action of $G \leq \Aut(S)$. Throughout we shall assume $g \geq 2$, so that $\Aut(S)$ is finite. Atiyah initially studied characteristics invariant under $f \in G$ \cite{Atiyah1971}, and recently a more comprehensive survey was made in \cite{Braden2025}, which amongst other results classified which Hurwitz curves with simple automorphism group had a unique invariant characteristic, though the parity of these invariant characteristics was not determined. Knowing the parity of invariant characteristics on a curve has been applied to, for example, determine $K$-stability of associated Fano threefolds \cite{Cheltsov2024}.

The reason that \cite{Braden2025} was able to prove existence of unique invariant characteristics on Hurwitz curves but not determine their parity is because of a disconnect between the two main current methods for studying theta characteristics. 
\begin{enumerate}
    \item Determine a particular symplectic homology basis and associated homology representation $G \to \mathrm{Sp}(2g, \mathbb{Z})$. The benefit of this method is that the full orbit structure of the theta characteristics and their parity may be computed \cite{Kallel2010}, but this is limited to lower genus because of current computational techniques. 
    \item Study the exact sequence of (abelian) groups associated to $S$
    \begin{equation}\label{eq: general exact sequence}
    0 \to \Hom(G, \mathbb{C}^\times) \to \Pic(G; S) \to \Pic(S)^G \to H^2(G, \mathbb{C}^\times) \to 0,
    \end{equation}
    where $\Pic(G; S)$ is the group of $G$-linearised line bundles, $\Pic(S)^G$ the group of $G$-invariant line bundles, and $H^2(G, \mathbb{C}^\times)$ is the Schur multiplier group \cite{Dolgachev1997a}. This method works well for curves where $S/G \cong \mathbb{P}^1$ and $G$ has well-known group properties, but apriori cannot determine the parity of invariant characteristics. 
\end{enumerate}

\begin{remark}
    Recall that a line bundle $L \in \Pic(S)^G$ is said to be \textit{$G$-linearised} if the associated collection of isomorphisms $\pbrace{\phi_h : L \overset{\cong}{\to} L \, | \, h \in G}$ forms a linear representation of $G$ as opposed to just a projective representation (which it must necessarily form for $L$ to be $G$-invariant). 
\end{remark}

New efficient computational tools have increased the range of genera for which approach (1) is useful \cite{Broughton2026}. In that paper the family of modular curves $X(p)$ ($p$ an odd prime, $p \geq 7$) was considered, for which approach (2) may be used to prove the existence of a unique invariant characteristic. Given that the parity of this characteristic is even for $p=7, 11, 13, 17$ (determined through approach (1)), the following conjecture was made. 

\begin{conjecture}[\cite{Broughton2026}]\label{conjecture}
    Let $p$ be an odd prime, $p\geq 7$. The unique invariant theta characteristic on the modular curve $X(p)$ is always even. 
\end{conjecture}

We shall prove this conjecture by demonstrating a method which uses the representation theory of the double cover $\tilde{G} = 2 \cdot G$. The method will be sufficiently general to be able to determine the parity of the invariant theta characteristics on all Hurwitz curves with simple automorphism group. In particular we will prove the following. 

\begin{theorem}\label{thm: parity of UIC on simple Hurwitz}
    If $S$ is a Hurwitz curve with simple automorphism group and unique invariant characteristic $L$, then $L$ is even.
\end{theorem}

\section{Simplified exact sequence}\label{sec: simplified sequence}

\subsection{Group simplifications}

In order to make progress we will consider particular curves for which (\ref{eq: general exact sequence}) simplifies.

We shall assume that $G$ is a perfect group, i.e. $\psquare{G, G} = G$. The consequences of this are three-fold.
\begin{enumerate}
    \item $\Hom(G, \mathbb{C}^\times)=0$ if $G$ is perfect \cite{Huppert2025}. 
    \item If $G$ is perfect, $G$ has a universal central extension $\tilde{G}$ which is itself perfect, defined by the short exact sequence of groups \cite[\S6.9]{Weibel1995} 
    \[
    0 \to H^2(G, \mathbb{C}^\times) \to \tilde{G} \to G \to 0,
    \]
    universal in the sense that for any other central extension $0 \to Z \to Y \to G \to 0$ there exists a unique homomorphism $f:\tilde{G} \to Y$ in the diagram
    \begin{center}
    \begin{tikzcd}
        0\arrow[r] & H^2(G, \mathbb{C}^\times) \arrow[r]\arrow[d] & \tilde{G} \arrow[r]\arrow[d, "\exists ! f"] & G \arrow[r]\arrow[d, "="] & 0 \\ 0\arrow[r] & Z \arrow[r] & Y \arrow[r, "\pi"] & G \arrow[r] & 0.
    \end{tikzcd}
    \end{center}
    \item $G$ is perfect implies that any $L \in \Pic(S)^G$ is $\tilde{G}$-linearised, and the linearisation is unique \cite[Corollary 1.2]{Dolgachev1997a}. 
\end{enumerate}

Assuming that $S$ admits a $G$-invariant theta characteristic $L$, which we shall do for the remainder of the paper, this means that we now have a linear representation $\rho : \tilde{G} \to \mathrm{GL}(\Gamma(L))$. We shall be aiming to show that (in appropriate situations) the dimension of this representation is even.

We furthermore assume that $H^2(G, \mathbb{C}^\times) = C_2$, whereby $\tilde{G} = 2 \cdot G$, the double cover (using the notation of the Atlas). In such a situation $\tilde{G}$ contains a unique non-identity element which generates the centre, which we shall denote $z$. The key lemma we require is the following. 

\begin{lemma}\label{lemma: dim even}
    Let $\tilde{G}$ be a perfect group with (central) involution $z$ and $\rho$ a representation for which $\rho(z)=-I_n$. Then $\dim \rho$ is even. 
\end{lemma}
\begin{proof}
    Writing $n=\dim \rho$, we have that $\det \rho(z) = (-1)^n$, but $\det \circ \rho$ is itself a 1-dimensional representation of $\tilde{G}$ and so must be a representation of the abelianisation of $\tilde{G}$ which is trivial. As such $\det \rho(z) =1 \Rightarrow n$ is even.
\end{proof}

As a result of Lemma \ref{lemma: dim even}, we see that if we can ensure that $\rho(z)=-I_n$, then $L$ has even parity. We shall see in the next subsection one method to achieve this. 

\subsection{Action simplifications}

So far we have only constrained the possible groups $G$, and not the way they act on $S$. To state the following result we shall need a brief definition. 

\begin{definition}
    Denote with $Q_1, \dots, Q_r$, the branch points of the quotient map $\pi : S \to S/G$, and for any $P_i \in \pi^{-1}(Q_i)$ let $c_i \in \mathbb{Z}_{>1}$ be the order of the corresponding isotropy group. Moreover denote $g_0 = g(S/G)$. We say that the tuple $(g_0; c_1, \dots, c_r)$ is the \textit{signature} of the action of $G$ on $S$.
\end{definition}

\begin{prop}[\cite{Dolgachev1997a}, \cite{Braden2025}]\label{prop: free part}
    Suppose that $G$ acts on $S$ with signature $(0; c_1, \dots, c_r)$. Then $\Pic(G; S)$ has free part $\mathbb{Z}$, and the generator of this has underlying $G$-invariant line bundle isomorphism class $\Gamma$ determined by $K_S = F \Gamma$ where $F \in \mathbb{Z}$ is given by 
    \[
     F= \lcm(c_1, \dots, c_r) \pround{r-2-\sum_{i=1}^r \frac{1}{c_i}}.
    \] 
\end{prop}

In order to use Proposition \ref{prop: free part} we shall assume that $g_0=0$. The condition we shall want to ensure (for reasons made clear by Lemma \ref{lemma: dim even}) is that $\rho(z) = -I_n$. Now as $L \in \Pic(S)^G$ there is a projective representation $\sigma : G \to \mathrm{PGL}(\Gamma(L))$. This lifts to $\sigma^\prime$ a linear representation of $Y$, a central extension of $G$ by $\mathbb{C}^\times$ defined via the diagram \cite{Weibel1995}
\begin{center}
\begin{tikzcd}
    0\arrow[r] & \mathbb{C}^\times \arrow[r]\arrow[d, "="] & Y \arrow[r]\arrow[d, "\sigma^\prime"] & G \arrow[r]\arrow[d, "\sigma"] & 0 \\ 0\arrow[r] & \mathbb{C}^\times \arrow[r] & \mathrm{GL}(\Gamma(L)) \arrow[r, "\pi"] & \mathrm{PGL}(\Gamma(L)) \arrow[r] & 0.
\end{tikzcd}
\end{center}
Coupling this to the diagram that exists by the universality of $\tilde{G}$ we have two maps $\tilde{G} \to \mathrm{GL}(\Gamma(L))$, namely $\rho$ and $\sigma^\prime \circ f$. Using the uniqueness of the $\tilde{G}$-linearisation of $L$ (any two such linearisation differ only by an element of $\Hom(\tilde{G}, \mathbb{C}^\times)=0$ \cite{Dolgachev1997a}) we see that $\rho = \sigma^\prime \circ f$. As $z \in \tilde{G}$ is in the centre it maps to the subgroup $\mathbb{C}^\times \leq \mathrm{GL}(\Gamma(L))$ consisting of diagonal matrices, and so the linear representation of $z$ is diagonal. As such, given $z$ is an involution, to have $\rho(z)=-I_n$ it is sufficient to show $\rho(z) \neq I_n$. 

One way to achieve this is to make sure $\Gamma(L)$ does not admit a linear representation of $G$, which would otherwise be induced by $\rho$ when $\rho(z)=I_n$. In order to have that $L$ is not $G$-linearisable it is sufficient to have that $K_S \in \Pic(G; S)$ is not a square, and so sufficient to have $F(c_i) = 1 \mod 2$. Writing each $c_i = 2^{a_i} b_i$ for $b_i \in \mathbb{Z}_{\geq 1}$ odd, and letting $A = \max(a_i)$, one can check that
\[
F = \left \lbrace \begin{array}{cc}
    0, & A=0, \\
    \abs{\pbrace{i \, | a_i=A}}, & A \neq 0,
\end{array} \mod 2 \right. 
\]

The culmination of this discussion is that if we have surface $S$ of genus $g\geq2$ with a theta characteristic invariant under $G \leq \Aut(S)$ such that 
\begin{enumerate}
    \item $G$ is perfect,
    \item $\abs{H^2(G, \mathbb{C}^\times)}=2$, 
    \item $g_0=0$, and 
    \item\label{item: parity} $A \neq 0$ with $\abs{\pbrace{i \, | a_i=A}}$ odd,
\end{enumerate}
then the invariant characteristic is even.

\section{Application to known curves}

\subsection{\secmath{A_5}}

The smallest perfect group is $G=A_5$, which also has $H^2(G, \mathbb{C}^\times) = C_2$ corresponding to the double cover $2 \cdot A_5 \to A_5$. We consider two actions of this group on curves of low genus. 
\begin{enumerate}
    \item $A_5$ acts with signature $(0; 2, 5, 5)$ on a surface $S$ of genus 4: this is Bring's curve, which has full automorphism $S_5$ \cite{Braden2023a}. Bring's curve is known to have a unique even characteristic invariant under the action of $A_5$ (indeed also under the full $S_5$), and the method laid out in \S\ref{sec: simplified sequence} again verifies this fact.
    \item $A_5$ acts with signature $(0; 3, 3, 5)$ on a hyperelliptic surface of genus 5. This curve admits a unique invariant characteristic which is odd \cite{Braden2025}, demonstrating that the condition (\ref{item: parity}) on the parity of $F$ as determined by $A$ is not vacuous. 
\end{enumerate}

\subsection{\secmath{X(p)}}

In the case considered in Conjecture \ref{conjecture}, $S=X(p)$ has full automorphism group $G=\mathrm{PSL}(2, p)$, which is simple and so necessarily perfect, acting with signature $(0; 2, 3, p)$ on $S$. Moreover, $H^2(G, \mathbb{C}^\times)=C_2$ corresponding to the double cover $\mathrm{SL}(2, p) \to \mathrm{PSL}(2, p)$. As such, all the necessary conditions are satisfied, and the conjecture of \cite{Broughton2026} is proven.

\subsection{Hurwitz curves with simple automorphism group}

Hurwitz curves are those with maximal automorphism group for their genus, namely having $\abs{\Aut(S)} = 84(g-1)$, which enforces that the signature of the action is $(0; 2, 3, 7)$. Hurwitz curves have perfect automorphism group \cite{Conder1990}, and those with simple automorphism group have been fully classified: \cite{Braden2025} identified which of these have a unique invariant characteristics, and in all cases $H^2(G, \mathbb{C}^\times) = C_2$. As such, we have proven Theorem \ref{thm: parity of UIC on simple Hurwitz}. 


\bibliography{library}

@InCollection{Dolgachev1997a,
  author    = {Dolgachev, I. V.},
  booktitle = {Recent progress in algebra. Proceedings of an international conference, KAIST, Taejon, South Korea, August 11--15, 1997},
  publisher = {American Mathematical Society},
  title     = {Invariant stable bundles over modular curves ${X}(p)$},
  year      = {1999},
  isbn      = {0-8218-0972-5},
  pages     = {65--99},
}

@Article{Braden2025,
  author    = {Braden, H. W. and Disney-Hogg, L.},
  journal   = {Experimental Mathematics},
  title     = {Orbits of Theta Characteristics},
  year      = {2025},
  issn      = {1058-6458},
  pages     = {1--50},
  comment   = {doi: 10.1080/10586458.2025.2481271},
  doi       = {10.1080/10586458.2025.2481271},
  publisher = {Taylor & Francis},
  url       = {https://doi.org/10.1080/10586458.2025.2481271},
}

@Misc{Broughton2026,
  author        = {Broughton, S. Allen and Disney-Hogg, Linden},
  howpublished  = {arXiv:2606.13922},
  title         = {Explicit homology representation for finite groups acting on {R}iemann surfaces},
  year          = {2026},
  archiveprefix = {arXiv},
  eprint        = {2606.13922},
  primaryclass  = {math.AG},
}

@Article{Atiyah1971,
  author    = {Atiyah, M. F.},
  journal   = {Annales scientifiques de l'\'Ecole Normale Sup\'erieure, Serie 4},
  title     = {Riemann surfaces and spin structures},
  year      = {1971},
  number    = {1},
  pages     = {47--62},
  volume    = {4},
  doi       = {10.24033/asens.1205},
  publisher = {Elsevier},
  url       = {www.numdam.org/item/ASENS_1971_4_4_1_47_0/},
}

@Misc{Cheltsov2024,
  author        = {I. Cheltsov and O. Li and S. Ma'u and A. Pinardin},
  howpublished  = {arXiv:2404.07803},
  title         = {K-stability and space sextic curves of genus three},
  year          = {2024},
  archiveprefix = {arXiv},
  eprint        = {2404.07803},
  primaryclass  = {math.AG},
}

@Article{Johnson1980,
  author  = {Johnson, D.},
  journal = {Journal of the London Mathematical Society},
  title   = {Spin Structures and Quadratic forms on Surfaces},
  year    = {1980},
  issn    = {0024-6107},
  number  = {2},
  pages   = {365--373},
  volume  = {s2-22},
  doi     = {10.1112/jlms/s2-22.2.365},
  url     = {https://doi.org/10.1112/jlms/s2-22.2.365},
}

@Article{Kallel2010,
  author  = {S. Kallel and D. Sjerve},
  journal = {Annales de la Facult{\'e} des Sciences de Toulouse},
  title   = {Invariant Spin Structures on {R}iemann Surfaces},
  year    = {2010},
  pages   = {457-477},
  volume  = {19},
}

@Article{Conder1990,
  author  = {Conder, M.},
  journal = {Bulletin (New Series) of the American Mathematical Society},
  title   = {Hurwitz groups: A brief survey},
  year    = {1990},
  number  = {2},
  pages   = {359--370},
  volume  = {23},
  url     = {https://doi.org/},
}

@Article{Braden2023a,
  author   = {Braden, H. W. and Disney-Hogg, L.},
  journal  = {European Journal of Mathematics},
  title    = {Bring{'}s curve: old and new},
  year     = {2023},
  issn     = {2199-6768},
  number   = {1},
  pages    = {3},
  volume   = {10},
  doi      = {10.1007/s40879-023-00706-0},
  refid    = {Braden2023},
  url      = {https://doi.org/10.1007/s40879-023-00706-0},
}

@Book{Weibel1995,
  author    = {Weibel, C. A.},
  publisher = {Cambridge University Press},
  title     = {An Introduction to Homological Algebra},
  year      = {1995},
  isbn      = {9780521559874},
  series    = {Cambridge Studies in Advanced Mathematics},
  volume    = {38},
  lccn      = {93015649},
  url       = {https://discovered.ed.ac.uk/permalink/44UOE_INST/7g3mt6/alma994944593502466},
}

@Book{Huppert2025,
  author    = {Huppert, B.},
  publisher = {Springer},
  title     = {Finite groups {I}. {Translated} from the {German} by {Christopher} {A}. {Schroeder}},
  year      = {2025},
  series    = {Grundlehren Math. Wiss.},
  volume    = {364},
  fseries   = {Grundlehren der Mathematischen Wissenschaften},
  issn      = {0072-7830},
}

@Article{Farkas2012,
  author   = {Farkas, G.},
  journal  = {Milan Journal of Mathematics},
  title    = {Theta Characteristics and Their Moduli},
  year     = {2012},
  issn     = {1424-9294},
  number   = {1},
  pages    = {1--24},
  volume   = {80},
  doi      = {10.1007/s00032-012-0178-7},
  refid    = {Farkas2012},
  url      = {https://doi.org/10.1007/s00032-012-0178-7},
}
\bibliographystyle{amsalpha}

\end{document}